\documentclass[a4paper,12pt]{amsart}
\usepackage[utf8]{inputenc}
\usepackage{amsmath,amssymb}
\usepackage{array,enumerate}
\usepackage[left=3cm,right=3cm,bottom=3cm,top=3cm]{geometry}
\usepackage{amsthm}
\usepackage{natbib}
\usepackage{hyperref,comment}
\usepackage{tikz}

\bibpunct{[}{]}{,}{n}{}{;}

\newtheorem{theorem}{Theorem}[section]
\newtheorem{lemma}[theorem]{Lemma}
\newtheorem{proposition}[theorem]{Proposition}

\newtheorem{corollary}[theorem]{Corollary}
\theoremstyle{definition}
\newtheorem{definition}[theorem]{Definition}
\newtheorem{example}[theorem]{Example}
\newtheorem{remark}[theorem]{Remark}

\newcommand{\SN}{\mathbb{N}}                    % Natural numbers
\newcommand{\SZ}{\mathbb{Z}}                    % Integers
\newcommand{\SC}{\mathbb{C}}                    % Complex numbers
\newcommand{\ra}[1]{\kern-1.5ex\xrightarrow{\ \ #1\ \ }\phantom{}\kern-1.5ex}
\newcommand{\ras}[1]{\kern-1.5ex\xrightarrow{\ \ \smash{#1}\ \ }\phantom{}\kern-1.5ex}

\allowdisplaybreaks[1]

\title{Enumeration of diagonally colored Young diagrams}
\author{Ádám Gyenge}
\address{Alfréd Rényi Institute of Mathematics, Hungarian Academy of Sciences, Reáltanoda utca 13-15, H-1053, Budapest, Hungary}
\email{gyenge.adam@renyi.mta.hu}

\begin{document}

\begin{abstract} In this note we give a new proof of a closed formula for the multivariable generating series of diagonally colored Young diagrams. This series also describes the Euler characteristics of certain Nakajima quiver varieties. Our proof is a direct combinatorial argument, based on Andrews' work on generalized Frobenius partitions. We also obtain representations of these series in some particular cases as infinite products.
\end{abstract}

\maketitle

%\tableofcontents

\section{Introduction}
\label{sec:intro}

Young diagrams are important combinatorial objects appearing in many areas of mathematics including representation theory, algebraic geometry and mathematical physics. As it is well known, Young diagrams correspond to partitions of natural numbers. Let us denote by $P(k)$ the number of partitions of the natural number $k$. Then the generating series of the numbers $P(k)$ is given by the classic result of Euler:
\begin{equation} \label{eq:partsum}
\sum_{k=0}^\infty P(k)q^k=\prod_{m=1}^{\infty}\frac{1}{1-q^m}\;.
\end{equation}

Many problems in representation theory and algebraic geometry require the enumeration of \emph{colored} Young diagrams. These are Young diagrams such that the boxes are labelled by numbers according to some specific rule.

In this note, we restrict our attention to diagonally colored Young diagrams. Theorem \ref{thm:coldiaggen} below gives their multivariable generating function in a closed form, giving a direct generalization of (\ref{eq:partsum}).  This generating series can be computed in different ways, for example using ideas related to representation theory \cite{fujii2005combinatorial}. Here we present a simple and elementary calculation based purely on combinatorics, following the ideas of Andrews \cite{andrews1984generalized}. Moreover, we obtain representations of these generating series in some particular cases as infinite products.

%In this note we restrict our attention to diagonally colored Young diagrams. The multivariable generating series of these was first obtained in \cite{fujii2005combinatorial} using representation theoretic and combinatorial methods. The result can be considered as a far-reaching generalization of (\ref{eq:partsum}). Here we present an alternative, elementary calculation of the multivariable generating series of diagonally colored Young diagrams.

\subsection{Diagonally colored Young diagrams}

A partition of a natural number $k \in \SN$ is a sequence of non-increasing natural numbers $\lambda_1 \geq \dots \geq \lambda_m > 0$ such that $\sum_i \lambda_i=k$. The Young diagram corresponding to a partition $\lambda=(\lambda_1, \dots,\lambda_m )$ is a collection of columns of square blocks with $\lambda_i$ boxes in the $i$-th column for $i=1,\dots,m$, such that the blocks in a column form a consecutive series without a hole. We identify partitions with the corresponding Young diagrams.  For a general reference on Young diagrams and on many of their numerous applications we refer the reader to \cite{fulton1997young, macdonald1998symmetric}.

The \emph{weight} $|Y|$ of a Young diagram $Y$ is the number of blocks in $Y$. In particular, if $Y$ corresponds to a partition $\lambda$, then $|Y|=\sum_i \lambda_i$, the number of which $\lambda$ is a partition of. The set of all Young diagrams will be denoted as $\mathcal{P}$. It decomposes into a disjoint union 
\[ \mathcal{P}=\bigsqcup_{k\geq 0} \mathcal{P}(k)\;,\]
where $\mathcal{P}(k)=\{ Y \in \mathcal{P} \colon |Y|=k \}$.

The block $s \in Y$ at the $i$-th row and $j$-th column is called the $(i,j)$-component of $Y$. Let $C$ be a set. A \emph{coloring} (or \emph{labeling}) of a Young diagram $Y$ with $C$ is a function assigning an element of $C$ to each box of $Y$.

In the rest of this section we fix an integer $n\geq 2$, the \emph{modulus}, and let $C=\SZ/n\SZ$.
\begin{definition} Let $Y \in \mathcal{P}$, and suppose we are given an $a \in C$. The \emph{diagonal $a$-coloring} (or, shortly, \emph{$a$-coloring}) of $Y$ is defined by associating to the $(i,j)$-component $s$ of $Y$ the residue
\[ \mathrm{res}(s)=a-i+j+n\SZ \in C\;.\]
\end{definition}
For any $a, c \in C$ and any diagram $Y \in \mathcal{P}$ we will denote by $\mathrm{wt}_c(a,Y)$, the \emph{$c$-weight of $Y$}, which is the number of boxes in $Y$, whose color according to the $a$-coloring is $c$. Clearly, $\sum_{c \in C} \mathrm{wt}_c(a, Y)=|Y|$ for any $a \in C$.
%We will call it just as the $a$-coloring of $Y$.

These notions can be extended to tuples of Young diagrams. Let $\underline{Y}=(Y_1,\dots,Y_l)$ be an $l$-tuple of Young diagrams. The \emph{total weight} of $\underline{Y}$ is defined as $|\underline{Y}|=\sum_{i=1}^l|Y_i|$. The set $\mathcal{P}_l$ of all $l$-tuples of Young diagrams decomposes as
\[ \mathcal{P}_l=\bigsqcup_{k\geq 0} \mathcal{P}_l(k)\;,\]
where $\mathcal{P}_l(k)=\{ \underline{Y} \in \mathcal{P}_l \colon |\underline{Y}|=k \}$. 

\begin{definition} Let $\underline{Y} \in \mathcal{P}_l$, and suppose we are given $a_m \in C$ for each $1\leq m \leq l$, which are collected in a vector $\underline{a}=(a_1,\dots,a_l)$.  The \emph{diagonal $\underline{a}$-coloring} (or, shortly, \emph{$\underline{a}$-coloring}) of $\underline{Y}$ is defined by associating to the $(i,j)$-component $s$ of $Y_m$ the residue
\[ \mathrm{res}(s)=a_m-i+j+n\SZ \in C\;.\]
\end{definition}
\begin{example} 
\label{ex:1}
For $n=3$ the diagonal $(2,1)$-coloring on the Young diagrams corresponding to the pair of partitions $((4,3,2),(2,1,1,1))$ is the following:
\[
\begin{pmatrix}\;\;
\begin{tikzpicture}[scale=0.6, font=\footnotesize, fill=black!20]
\draw (0, 0) -- (3,0);
\draw (0,1) --(3,1);
\draw (0,2) --(3,2);
\draw (0,3) --(2,3);
\draw (0,4) --(1,4);
\draw (0,0) -- (0,4);
\draw (1,0) -- (1,4);
\draw (2,0) -- (2,3);
\draw (3,0) -- (3,2);
\draw (0.5,0.5) node {2};
\draw (1.5,0.5) node {0};
\draw (2.5,0.5) node {1};
\draw (0.5,1.5) node {1};
\draw (1.5,1.5) node {2};
\draw (2.5,1.5) node {0};
\draw (0.5,2.5) node {0};
\draw (1.5,2.5) node {1};
\draw (0.5,3.5) node {2};
\end{tikzpicture},\;
\begin{tikzpicture}[scale=0.6, font=\footnotesize, fill=black!20]
\draw (0, 0) -- (4,0);
\draw (0,1) --(4,1);
\draw (0,2) --(1,2);
\draw (0,2) --(0,0);
\draw (1,2) --(1,0);
\draw (2,1) -- (2,0);
\draw (3,1) -- (3,0);
\draw (4,1) -- (4,0);
\draw (0.5,0.5) node {1};
\draw (1.5,0.5) node {2};
\draw (2.5,0.5) node {0};
\draw (3.5,0.5) node {1};
\draw (0.5,1.5) node {0};
\end{tikzpicture}\;\;
\end{pmatrix}\;.
\]
\end{example}

Naturally, for any $\underline{a} \in C^l$, $c \in C$ and any $l$-tuple of Young diagrams $\underline{Y} \in \mathcal{P}_l$ we will denote by $\mathrm{wt}_c(\underline{a},\underline{Y})$, the \emph{$c$-weight of $\underline{Y}$}, which is the number of boxes in $\underline{Y}$, whose color according to the $\underline{a}$-coloring is $c$. Clearly, $\sum_{c \in C} \mathrm{wt}_c(\underline{a}, \underline{Y})=|\underline{Y}|$ for any $\underline{a} \in C^l$. We arrange the $c$-weights into a vector $\underline{\mathrm{wt}}(\underline{a},\underline{Y})=(\mathrm{wt}_0(\underline{a},\underline{Y}),\dots, \mathrm{wt}_{n-1}(\underline{a},\underline{Y})) \in (\SZ_{\geq 0})^{n}$.

Using these notations for any fixed $\underline{a} \in C^l$ the set $\mathcal{P}_l(k)$ decomposes as
\[ \mathcal{P}_l(k)=\bigsqcup_{|\underline{v}|=k}\mathcal{P}_{\underline{a}}(\underline{v})\;, \]
where $\underline{v} \in (\SZ_{\geq 0})^{n}$, $|\underline{v}|=\sum_c v_c$ and 
\[ \mathcal{P}_{\underline{a}}(\underline{v})=\{ \underline{Y} \in \mathcal{P}_l(k)\;|\; \underline{\mathrm{wt}}(\underline{a},\underline{Y})=\underline{v} \}\;.\]

\subsection{Generating series}

Let us introduce formal variables $q_c$ associated to each color $c \in C$. These will be collected into a vector $\underline{q}=(q_0,\dots, q_{n-1})$. The indices are always meant as elements in $C$. We introduce the notation $\underline{q}^{\underline{k}}=\prod_{c \in C} q_c^{k_c}$ for any vector $\underline{k}$ whose components are indexed by $C$. In particular, $\underline{q}^{\underline{\mathrm{wt}}(a,Y)}=\prod_{c \in C} q_c^{\mathrm{wt}_c(a,Y)}$ and $\underline{q}^{\underline{\mathrm{wt}}(\underline{a},\underline{Y})}=\prod_{c \in C} q_c^{\mathrm{wt}_c(\underline{a},\underline{Y})}$. Throughout the article we always assume that, when evaluated, $|q_c| \ll 1$ for all $c \in C$. 

The colored (multivariable) generating series of colored Young diagrams is defined as
\[ Z_a(\underline{q})=\sum_{Y \in  \mathcal{P}} \underline{q}^{\underline{\mathrm{wt}}(a,Y)}\;. \]
Similarly, the colored (multivariable) generating series of $l$-tuples of colored Young diagrams is defined as
\[
Z_{\underline{a}}(\underline{q})=\sum_{\underline{Y} \in  \mathcal{P}_l} \underline{q}^{\underline{\mathrm{wt}}(\underline{a},\underline{Y})}\;. 
\]

Since $\mathcal{P}_l=\mathcal{P}^l$, and the coloring function is independent on the individual components, one immediately obtains
\begin{corollary}
\[ Z_{\underline{a}}(\underline{q})=\prod_{m=1}^l Z_{a_m}(\underline{q})\;. \]
\end{corollary}
In particular, the calculation of $Z_{\underline{a}}(\underline{q})$ is easily reduced to that of $Z_a(\underline{q})$.

\subsection{Euler characteristics of quiver varieties}

We now make a digression, which emphasizes the importance of the series $Z_{\underline{a}}(\underline{q})$.  %These quiver varieties are in turn the moduli spaces of equivariant torsion-free sheaves on $\SC^2$ with respect to the action of the group $\SZ/n\SZ$.

Let $(I,H)$ be a quiver. More precisely, $I$ is a set of vertices and $H$ is a set of oriented edges. Let $\overline{H}=H \cup H^{\ast}$, where $H^{\ast}$ is the set of edges in $H$ with the reversed orientation. %Assume we are given a subset $\Omega \subset H$ such that $\Omega \cup \overline{\Omega}=H$, and $\Omega \cap \overline{\Omega}=\emptyset$ where $\overline{\phantom{a}}$ means reversing orientation of edges. 
For dimension vectors $\underline{v}, \underline{w} \in \SZ^I_{\geq 0}$ %and $\zeta=(\zeta_{\SC},\zeta_{\SR}) \in \SC^I \oplus \SR^I$ 
we define a Nakajima quiver variety as follows. See \cite{nakajima2002geometric} and references therein for the details, here we follow the notations of \cite{sam2014combinatorial}. %For simplicity we assume that $\underline{v} \in \SZ^I_{> 0}$, since otherwise a slight and non-essential modification is needed.

Fix $I$-graded vector spaces $V , W$ such that $\dim V_i = v_i, \dim W_i = w_i$. Let
\[ M(\underline{v},\underline{w}) =\left( \bigoplus_{h \in \overline{H}} \mathrm{Hom}(V_{\mathrm{s}(h)},V_{\mathrm{t}(h)})\right)\oplus \left( \bigoplus_{i \in I} \mathrm{Hom}(W_i,V_i)\oplus \mathrm{Hom}(V_i,W_i) \right), \]
where $h \in \overline{H}$ is an oriented edge from $\mathrm{s}(h)$ to $\mathrm{t}(h)$. Note that  $GL(V)=\prod GL(V_i)$ acts on $M(\underline{v},\underline{w})$ by
\[ (g_i) \cdot (B_h,a_i,b_i)=(g_{\mathrm{t}(h)}B_h g_{\mathrm{s}(h)}^{-1},g_ia_i, b_ig_i^{-1}) \]
for any $g_i \in GL(V_i)$, $B_h \in \mathrm{Hom}(V_{\mathrm{s}(h)},V_{\mathrm{t}(h)})$, $a_i \in \mathrm{Hom}(W_i,V_i)$ and $b_i \in \mathrm{Hom}(V_i,W_i)$. Elements in $M(\underline{v},\underline{w})$ can be shortly denoted as a triple $(B,a,b)$. Here $B$ is a representation of the path algebra of the quiver on $V$, while $a\colon W \to V$ and $b\colon V \to W$ are maps of $I$-graded vector spaces.

The moment map $\mu$ for the $GL(V)$-action on $M(\underline{v},\underline{w})$ is given by
\[ \mu (B,a,b)= \bigoplus_{i \in I}\left( \sum_{h:\mathrm{t}(h)=i} \epsilon(h)B_hB_{h^{\ast}}+a_ib_i \right) \in \bigoplus_{i \in I}\mathfrak{gl}(V_i)=\mathfrak{gl}(V) \;,\]
where $\epsilon(h)=1$ and $\epsilon(h^{\ast})=-1$ for $h \in H$. 
%Then $\mu^{-1}(0)$ is invariant under the action of $GL_{\underline{v}}$.
A triple $(B,a,b) \in M(\underline{v},\underline{w})$ is called \emph{stable} if $\mathrm{im}(a)$ generates $V$ under the action of $B$. The subset of stable triples in $M(\underline{v},\underline{w})$ is denoted as $M(\underline{v},\underline{w})^{\mathrm{st}}$.

The quiver variety associated to the dimension vectors $\underline{v}, \underline{w}$ is
\[\mathcal{M}(\underline{v},\underline{w})=\{ (B,a,b) \in M(\underline{v},\underline{w})^{\mathrm{st}} \;|\; \mu(B,a,b)=0 \}/ GL(V)\;. \]
This is well defined only up to a (non-canonical) isomorphism, but since we are only interested in its topological properties we do not consider its dependence on the vector spaces $V$ and $W$ here.

Let $(I, H)$ be the affine Dynkin quiver $A^{(1)}_{n-1}$ with the cyclic orientation of the edges, hence $I$ can be identified with $C=\SZ/n\SZ$:
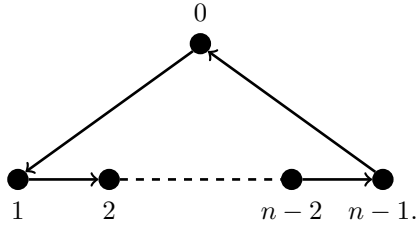
\begin{center}
\begin{tikzpicture}[scale=1.2,line width=1pt, font=\footnotesize]
\node (1)  at ( 0,0) [circle,draw, fill=black, inner sep=0pt,minimum size=7pt, label=below:$1$] {};
\node (2) at ( 1,0) [circle,draw, fill=black, inner sep=0pt,minimum size=7pt,label=below:$2$] {};
\node (n-2) at ( 3,0) [circle,draw, fill=black, inner sep=0pt,minimum size=7pt,label=below:$n-2$] {};
\node (n-1) at ( 4,0) [circle,draw,fill=black,inner sep=0pt,minimum size=7pt, label=below:$n-1$.] {};
\node (0) at ( 2,1.5) [circle,draw,fill=black,inner sep=0pt,minimum size=7pt, label=above:$0$] {};
\draw [->] (1.east) -- (2.west);
\draw [dashed] (2.east) -- (n-2.west);
\draw [->] (n-2.east) -- (n-1.west);
\draw [->] (n-1.north west) -- (0.south east);
\draw [->] (0.south west) -- (1.north east);
\end{tikzpicture}
\end{center}
It can be shown that there is a $T=(\SC^\ast)^{|\underline{w}|+2}$-action on the associated quiver variety $\mathcal{M}(\underline{v},\underline{w})$, whose fixed points are isolated. Let $\underline{a}=(0,\dots,0,\dots,n-1, \dots,n-1)$, where for each  $c \in C$ the number of $c$'s in $\underline{a}$ is $w_c$. Elements of $\underline{a}$ correspond in turn to basis vectors of $W$. The exact order of the entries is not important for us since a permutation of them corresponds to an automorphism of $W$ on which, as mentioned above, the topology of $\mathcal{M}(\underline{v},\underline{w})$ does not depend.

\begin{proposition}{\cite[Proposition 5.7]{sam2014combinatorial}} The $T$-fixed points of $\mathcal{M}(\underline{v},\underline{w})$ are indexed by $|\underline{w}|$-tuples of diagonally colored Young diagrams $\underline{Y}$ such that $|\underline{Y}|=|\underline{v}|$, the $i$-th diagram $Y_i$ in $\underline{Y}$ is given the $a_i$-coloring, and $\underline{\mathrm{wt}}(\underline{a},\underline{Y})=\underline{v}$. %Moreover, taking into account the weights of the torus action each fixed point gets a coloring, such that the $i$-th part is given the diagonal $w_i$-coloring.
\end{proposition}

\begin{corollary} For any $w$ fixed, and $\underline{a}$ as above,
\[\sum_{\underline{v}}\chi(\mathcal{M}(\underline{v},\underline{w}))\underline{q}^{\underline{v}}=Z_{\underline{a}}(\underline{q})\;.\]
\end{corollary}

\subsection{The results}

%The key players in the proof are the generalized Frobenius partitions introduced in where they were used to calculate the generating series...

For the sequel we will assume that $l=1$, hence we work only with partitions. In Section 2 we generalize some results of Andrews on F-partitions developed in \cite{andrews1984generalized}. Using these in Section 3 we present an elementary proof of the following fact, which can also be read off from the results of \cite{fujii2005combinatorial} or , proved there by a completely different method based on abacus combinatorics:
\begin{theorem}
\label{thm:coldiaggen}
\[ Z_a(\underline{q})= \left( \prod_{m=1}^\infty (1-q^m)^{-1} \right)^{n} \cdot\sum_{ \underline{m}=(m_1,\dots,m_{n-1}) \in \SZ^{n-1} } q_{1+a}^{m_1}\cdot\dots\cdot q_{n-1+a}^{m_{n-1}}(q^{1/2})^{\underline{m}^\top \cdot C \cdot \underline{m}}\;,\]
where $q=\prod_{i=0}^{n-1} q_i$, and $C$ is the Cartan matrix of finite type $A_{n-1}$.
\end{theorem}
%\begin{remark}

%\end{remark}

Finally, in Section 4 we obtain the following representations:
\begin{corollary}
\label{cor:n23}
\begin{enumerate}
\item For $n=2$ and $a=0$,
\begin{equation}
\label{eq:inf2}
Z_0(\underline{q})=\prod_{m=1}^\infty \frac{(1+q_{1}q^{2m-1})(1+q_{0}q^{2m-2})}{(1-q^m)(1-q^{2m-1})}\;.
\end{equation}
\item For $n=3$ and $a=0$,
\begin{multline}
\label{eq:inf3}
Z_0(\underline{q})=\\ \prod_{m=1}^\infty \frac{(1-q^{6m})(1+q_1q_2^2q^{6m-3})(1+q_0^2q_1q^{6m-5})(1+q_1q^{2m-1})(1+q_0q_2q^{2m-2})}{(1-q^m)^2(1-q^{2m-1})} \\
+q_0\cdot \\ \prod_{m=1}^\infty \frac{(1-q^{6m})(1+q_1q_2^2q^{6m-6})(1+q_0^2q_1q^{6m-1})(1+q_1q^{2m-2})(1+q_0q_2q^{2m-1})}{(1-q^m)^2(1-q^{2m-1})}
\;.
\end{multline}
\end{enumerate}
\end{corollary}
\section{Generalized Frobenius partitions}

\subsection{Uncolored case}
\begin{definition} Two rows of nonnegative integers
\[
\begin{pmatrix}
f_1 & f_2 & \dots & f_d \\
g_1 & g_2 & \dots & g_d
\end{pmatrix} \]
are called a \emph{generalized Frobenius partition} or \emph{F-partition} of $k$ if \[k=d + \sum_{i=1}^d (f_i+g_i)\;.\]
\end{definition}
\begin{remark} A generalized Frobenius partition is a classical Frobenius partition if moreover $f_1 > f_2 > \dots > f_d \geq 0$ and $g_1 > g_2 > \dots > g_d \geq 0$. In this case, we can associate to the F-partition a Young diagram from which if we delete the $d$ long diagonal then the lengths of the rows below it are $f_1$, $f_2$, etc. and the length of the columns above the diagonal are $g_1$,  $g_2$, etc. This correspondence between Young diagrams and classical F-partitions is bijective.
\end{remark}

Let $H$ be an arbitrary set consisting of finite sequences of nonnegative integer. For arbitrary integers $d$ and $k$ let $P_H(k,d)$ denote the sequences in $H$ of length $d$ which sum up to $k$. For any pair of such sets $H_1$ and $H_2$, let moreover $P_{H_1,H_2}(k)$ be the number of generalized Frobenius partitions of $k$ with elements in the first row $(f_1,\ldots, f_d)$ from $H_1$ and with elements in the second row $(g_1,\ldots, g_d)$ from $H_2$. Then the very useful result of Andrews says the following
\begin{theorem}[\cite{andrews1984generalized}, Section 3]
\label{thm:genfrobgs}
\[ \sum_{k=0}^\infty P_{H_1,H_2}(k)q^k=[z^0] \sum_{k,m}P_{H_1}(k,d)q^k(zq)^d\sum_{k,d}P_{H_2}(k,d)q^kz^{-d}\;, \]
where $[z^m]\sum A_kz^k=A_m$.
\end{theorem}

The term $q^d$ in the first term of the right hand size corresponds to the contribution of the diagonals. To have a more symmetric formula we will slightly change the notions. Transform each generalized Frobenius partition
\[
\begin{pmatrix}
f_1 & f_2 & \dots & f_d \\
g_1 & g_2 & \dots & g_d
\end{pmatrix} \]
into
\[
\begin{pmatrix}
f_1+1 & f_2+1 & \dots & f_d+1 \\
g_1 & g_2 & \dots & g_d
\end{pmatrix} \;.\]
Then\[k=\sum_{i=1}^d ((f_i+1)+g_i)\;.\]

For $H$ an arbitrary set of sequences as above, let $H'=\{ (f_1+1,\dots,f_d+1) \;:\; (f_1,\dots,f_d) \in H \}$ be the \emph{shift} of $H$ by one upward. In particular, elements of $H'$ consist of sequences of strictly positive integers. Then $P_{H'}(k,d)=P_{H}(k-d,d)$, and
\begin{equation} 
\label{eq:genfrobgsmod}
\sum_{k=0}^\infty P_{H'_1,H_2}(k)q^k=[z^0] \sum_{k,d}P_{H'_1}(k,d)q^{k}z^d\sum_{k,d}P_{H_2}(k,d)q^kz^{-d}\;. 
\end{equation}
The advantage of \eqref{eq:genfrobgsmod} is that it can be used in a more general context. Namely, elements of $H'_1$ and $H_2$ can be both arbitrary sequences of nonnegative entries.

\subsection{Colored case}
\label{subsec:colfpart}

We aim for a multivariable generalization of Theorem \ref{thm:genfrobgs}. Consider first an arbitrary finite coloring set $C$ and let $\underline{k}$ be a vector of nonnegative integers indexed by the elements of $C$.
\begin{definition}
Two series of vectors consisting of integers and arranged into two rows as
\[
\begin{pmatrix}
\underline{f}_1 & \underline{f}_2 & \dots & \underline{f}_d \\
\underline{g}_1 & \underline{g}_2 & \dots & \underline{g}_d
\end{pmatrix} \]
are called a \emph{colored generalized Frobenius partition} or \emph{colored F-partition} of $\underline{k}$ if
\begin{enumerate}
\item the elements in $\underline{f}_i$ and $\underline{g}_i$ are indexed by the elements $c \in C$ for each $1\leq i\leq d$;
\item $f_{i,c} \geq 0$ and $g_{i,c} \geq 0$ for every $1\leq i\leq d$ and $c \in C$;
\item $\sum_{i=1}^d (f_{i,c} + g_{i,c})=k_c$ for each $c \in C$.
\end{enumerate}
We will call $k=\sum_{c\in C} k_c=\sum_{c\in C} \sum_{i=1}^d (f_{i,c} + g_{i,c})$ the \emph{total weight} of such a colored F-partition.
\end{definition}
At the moment we do not require any further relations between the elements $f_{i,c}$ and $g_{i,c}$ but see \autoref{sec:coldiagproof} and particularly \autoref{ex:2diag} below, where we apply this general construction to the enumeration of diagonally colored Young diagrams.

Let $H$ be an arbitrary set consisting of tuples of vectors, each of which is indexed by elements of $C$. For an arbitrary vector $\underline{k}$ indexed by the elements of $C$ and consisting of nonnegative integers let $P_{H}(\underline{k},d)$ be the number of $d$-tuples of vectors $(\underline{f}_1,\dots, \underline{f}_d) \in H$  which satisfy conditions (1) and (2) such that $\sum_{i=1}^{d} f_{i,c}=k_c$. 
If both $H_1$ and $H_2$ are sets consisting of tuples of vectors, each element of which is indexed by elements of $C$, then let  $P_{H_1,H_2}(\underline{k})$ be the number of colored F-partition of $\underline{k}$ in which the top row is in $H_1$ and the bottom row is in $H_2$. 

Then the same ideas as that of Theorem \ref{thm:genfrobgs} imply the following multivariable analogue of (\ref{eq:genfrobgsmod}).
\begin{theorem}
\label{thm:genfrobmgs}
\[ \sum_{\underline{k}} P_{H_1,H_2}(\underline{k})\underline{q}^{\underline{k}}=[z^0] \sum_{\underline{k},d}P_{H_1}(\underline{k},d)z^d\underline{q}^{\underline{k}}\sum_{\underline{k},d}P_{H_2}(\underline{k},d)z^{-d}\underline{q}^{\underline{k}}\;. \]
\end{theorem}

\section{Proof of Theorem \ref{thm:coldiaggen}}
\label{sec:coldiagproof}

We return to the setting of Section \ref{sec:intro}. We let $C=\mathbb{Z}/n\mathbb{Z}$ and fix an $a \in C$. To be able to apply Theorem \ref{thm:genfrobmgs} we first associate to each Young diagram $Y \in \mathcal{P}$ a colored F-partition of $\underline{k}=\{ \mathrm{wt}_c(a,Y) \}_{c \in C}$ which uniquely describes the diagonal $a$-coloring on $Y$. Assume that the main diagonal of $Y$ consists of $d$ blocks. Then the associated colored $F$-partition is
\[
\begin{pmatrix}
\underline{f}_1 & \underline{f}_2 & \dots & \underline{f}_d \\
\underline{g}_1 & \underline{g}_2 & \dots & \underline{g}_d
\end{pmatrix} \;,\]
where ${f_{i,c}}$ is the number of blocks of color $c$ in the $i$-th row below and including the main diagonal, and ${g_{i,c}}$ is the number of blocks of color $c$ in the $i$-th column above the main diagonal for every $1 \leq i \leq d$.
\begin{example} 
\label{ex:2diag}
The colored F-partition associated to the first diagram in Example \ref{ex:1} is
\[
\begin{pmatrix}
(1,1,1) & (1,0,1) \\
(1,1,1) & (0,1,0)
\end{pmatrix}\;, \]
where each $\underline{f}_i=(f_{i,0},f_{i,1},f_{i,2})$ and each $\underline{g}_i=(g_{i,0},g_{i,1},g_{i,2})$.
\end{example}

For $i=1,2$ let $H_i$ be the set of tuples of vectors which can appear as the $i$-th row of a colored F-partition associated to a diagonally $a$-colored Young diagram in the above construction. We omitted from the notation the dependence on $a \in C$. Then
\begin{equation} 
\label{eq:colgen}
Z_a(\underline{q})= \sum_{\underline{k}} P_{H_1,H_2}(\underline{k})\underline{q}^{\underline{k}}\;. 
\end{equation}

\begin{example}
For $a=0$, consider the following linearly ordered set of vectors:
\[ \{ (1,0,\dots,0) < (1,1,0,\dots,0) < \dots < (1,\dots,1) < (2,1,\dots,1) < \dots   \}.\]
The elements of each vector are indexed by the set $C=\{0,\dots,n-1\}$.
Then, in our setting $H_1$ consists of finite, decreasing sequences with elements from this ordered set. Similarly, $H_2$ consists of finite, decreasing sequences with elements from the ordered set
\[ \{ (0,\dots,0,0) < (0,\dots,0,1) < \dots < (1,\dots,1,1) < (1,\dots,1,2) < \dots   \}.\]
\end{example}

\begin{lemma}
\label{lem:fblinegen}
\begin{enumerate}
\item 
\[ \sum_{\underline{k}} P_{H_1}(\underline{k},d)z^d \underline{q}^{\underline{k}}= \prod_{k=0}^\infty\prod_{i=0}^{n-1}(1+zq_{0+a}\dots q_{i+a}q^k)\;. \]
\item 
\[ \sum_{\underline{k}} P_{H_2}(\underline{k},d)z^d \underline{q}^{\underline{k}}= \prod_{k=0}^\infty\prod_{i=0}^{n-1}(1+z^{-1}q_{i+1+a}\dots q_{n-1+a}q^k)\;. \]
\end{enumerate}
\end{lemma}
\begin{proof} (1) It is clear that each term $zq_{0+a}\dots q_{i+a}q^k$ corresponds to a part of a column above and including the main diagonal which has length $nk+i$. Conversely, the decomposition of each nonnegative number as $nk+i$ is unique.

The proof of (2) is similar.
\end{proof}

The product of the two generating series in Lemma \ref{lem:fblinegen} is
\begin{equation}
\label{eq:prodgen}
\begin{gathered}
\sum_{\underline{k}} P_{H_1}(\underline{k},d)z^d \underline{q}^{\underline{k}} \cdot \sum_{\underline{k}} P_{H_2}(\underline{k},d)z^d \underline{q}^{\underline{k}} \\= \prod_{k=0}^\infty\prod_{i=0}^{n-1}(1+zq_{0+a}\dots q_{i+a}q^k)(1+z^{-1}q_{i+1+a}\dots q_{n-1+a}q^k)\\
= \prod_{k=1}^\infty\prod_{i=0}^{n-1}(1+zq_{i+1+a}^{-1}\dots q_{n-1+a}^{-1}q^k)(1+(zq_{i+1+a}^{-1}\dots q_{n-1+a}^{-1})^{-1}q^{k-1})\\
=\left( \prod_{m=1}^\infty (1-q^m)^{-1} \right)^{n} \prod_{i=0}^{n-1} \left( \sum_{j_i=-\infty}^{\infty} (zq_{i+1+a}^{-1}\dots q_{n-1+a}^{-1})^{j_i}q^{\binom{j_i+1}{2}}\right)\;,
\end{gathered}
\end{equation}
where at the last equality we have used the following form of the Jacobi triple product formula:
\[ \prod_{n=1}^{\infty}(1+zq^n)(1+z^{-1}q^{n-1})= \left(\prod_{n=1}^{\infty}(1-q^n)^{-1}\right)\sum_{j=-\infty}^{\infty}z^jq^{\binom{j+1}{2}}\;.\]

By (\ref{eq:colgen}) and Theorem \ref{thm:genfrobmgs} to obtain $Z_a(\underline{q})$ we have to calculate the coefficient of $z^{0}$ in (\ref{eq:prodgen}).

\begin{equation}
\label{eq:z0coeff}
\begin{array}{rcl}
Z_a(\underline{q})&=&[z^0]\Big(\left( \prod_{m=1}^\infty (1-q^m)^{-1} \right)^{n}\cdot \\ & & \prod_{i=0}^{n-1} \left( \sum_{j_i=-\infty}^{\infty} (zq_{i+1+a}^{-1}\dots q_{n-1+a}^{-1})^{j_i}q^{\binom{j_i+1}{2}}\right)\Big)\\
%&=\left( \prod_{m=1}^\infty (1-q^m)^{-1} \right)^{n} \prod_{i=0}^{n-1} \left( \sum_{m_i=-\infty}^{\infty} (zq_{i+1+a}^{-1}\dots q_{n-1+a}^{-1})^{m_i}q^{\binom{m_i+1}{2}}\right)\\
&=& \left( \prod_{m=1}^\infty (1-q^m)^{-1} \right)^{n} \cdot \\ & &\sum_{ \substack{ \underline{j}=(j_0,\dots,j_{n-1}) \in \SZ^{n} \\ \sum_i j_i=0}} q_{1+a}^{-j_0}\cdot\dots\cdot q_{n-1+a}^{-j_0-\dots-j_{n-2}}q^{\sum_{i=0}^{n-1}\binom{j_i+1}{2}}\;.
\end{array}
\end{equation}

Let us introduce the following series of integers:
\begin{equation}
\label{eq:mjcorr}
\begin{array}{r c l}
m_1& = & -j_0\;, \\
m_2 & = & -j_0-j_1 \;, \\
& \vdots & \\
m_{n-1} & = & -j_0-j_1-\dots-j_{n-2} \;. \\
\end{array}
\end{equation}
It is obvious that the map
\[ 
\begin{array}{r c l}
\left\{ (j_0,\dots,j_{n-1}) \in \SZ^{n}\;:\; \sum_i j_i=0 \right\} & \rightarrow &\SZ^{n-1} \\
 \quad (j_0,\dots,j_{n-1}) &\mapsto& (m_1,\dots,m_{n-1})
\end{array}
\]
is a bijection. The inverse of it is
\[
\begin{array}{r c l}
j_0& = & -m_1\;, \\
j_1 & = & -m_2+m_1 \;, \\
& \vdots & \\
j_{n-2} & = & -m_{n-1}+m_{n-2} \;, \\
j_{n-1} & =& m_{n-1}\;.
\end{array}
\]

If $n=2$, then
\begin{equation}
\label{eq:trcart1}
\sum_{i=0}^{1}\binom{j_i+1}{2} = \binom{-m_1+1}{2} + \binom{m_{1}+1}{2} = m_1^2 = \frac{1}{2}\left(\underline{m}^\top \cdot C \cdot \underline{m}\right)\;,
\end{equation}
where $C=(2)$ is the Cartan matrix of type $A_{1}$.

If $n>2$, then
\begin{equation}
\label{eq:trcart2}
\begin{aligned}
\sum_{i=0}^{n-1}\binom{j_i+1}{2} & = \binom{-m_1+1}{2} + \sum_{i=1}^{n-2}\binom{-m_{i+1}+m_{i}+1}{2}+ \binom{m_{n-1}+1}{2} \\
& = m_1^2+\dots+m_{n-1}^2 -\sum_{i=1}^{n-2}m_im_{i+1} \\
&= \frac{1}{2}\left(\underline{m}^\top \cdot C \cdot \underline{m}\right)\;,
\end{aligned}
\end{equation}
where 
\[C=
\begin{pmatrix}
2 & -1 & & & & \\
-1 & 2 & -1 & & & \\
& -1 & 2 & & &\\
& & & \ddots & & \\
& & & & & -1\\
& & & & -1 & 2
\end{pmatrix}
\] is the Cartan matrix of type $A_{n-1}$.

Equations (\ref{eq:z0coeff}), (\ref{eq:trcart1}) and (\ref{eq:trcart2}) together immediately imply Theorem \ref{thm:coldiaggen} for all $n > 1$.

%\section{Final comments}

%As pointed out in \cite{fujii2005combinatorial} the series (\ref{eq:multgen}) is the generating series of of the euler

\section{Proof of Corollary \ref{cor:n23}}

\subsection{Proof of (\ref{eq:inf2})}

For $n=2$ and $a=0$,
\[\begin{aligned}Z_0(\underline{q}) & =  \frac{\sum_{m_1=-\infty}^{\infty}q_1^{m_1}q^{m_1^2}}{\prod_{m=1}^{\infty}(1-q^m)^2}\\ & = \frac{\prod_{m=1}^{\infty}(1-q^{2m})(1+q_1^{2m-1})(1+q_1^{-1}q^{2m-1})}{\prod_{m=1}^{\infty}(1-q^m)^2} \\
&  =\prod_{m=1}^{\infty}\frac{(1+q_1q^{2m-1})(1+q_0q^{2m-2})}{(1-q^m)(1-q^{2m-1})}\;,
\end{aligned}\]
where at the second equality we have used the following form of the Jacobi triple product identity:
\begin{equation} 
\label{eq:jactr2}
\prod_{n=1}^{\infty}(1-q^{2n})(1+zq^{2n-1})(1+z^{-1}q^{2n-1})= \sum_{j=-\infty}^{\infty}z^jq^{j^2}\;.
\end{equation}

\subsection{Proof of (\ref{eq:inf3})}

 For $n=3$ and $a=0$, the numerator of the generating series can be written as the sum
\[\sum_{m_1,m_2=-\infty}^{\infty}q_1^{m_1}q_2^{m_2}q^{m_1^2+m_2^2-m_1m_2}=\]
\begin{align*}
=&\sum_{m_1,m_2=-\infty}^{\infty}q_1^{m_1}q_2^{2m_2}q^{m_1^2+(2m_2)^2-2m_1m_2}\\ &+\sum_{m_1,m_2=-\infty}^{\infty}q_1^{m_1}q_2^{2m_2-1}q^{m_1^2+(2m_2-1)^2-m_1(2m_2-1)} \\
 =&\sum_{m_2=-\infty}^{\infty}q_1^{m_2}q_2^{2m_2}q^{3m_2^2}\sum_{m_1=-\infty}^{\infty}q_1^{m_1-m_2}q^{(m_1-m_2)^2} \\* &+\sum_{m_2=-\infty}^{\infty}q_1^{m_2}q_2^{2m_2-1}q^{3m_2^2-3m_2+1}\sum_{m_1=-\infty}^{\infty}q_1^{m_1-m_2}q^{(m_1-m_2)^2+m_1-m_2} \\
=&\sum_{m_2=-\infty}^{\infty}q_1^{m_2}q_2^{2m_2}q^{3m_2^2}\sum_{m_1=-\infty}^{\infty}q_1^{m_1}q^{m_1^2}\\* &+\sum_{m_2=-\infty}^{\infty}q_1^{m_2}q_2^{2m_2-1}q^{3m_2^2-3m_2+1}\sum_{m_1=-\infty}^{\infty}q_1^{m_1}q^{m_1^2+m_1}\\
=&\sum_{m_2=-\infty}^{\infty}q_1^{m_2}q_2^{2m_2}q^{3m_2^2}\sum_{m_1=-\infty}^{\infty}q_1^{m_1}q^{m_1^2} \\*
&+q_2^{-1}q\sum_{m_2=-\infty}^{\infty}q_1^{m_2}q_2^{2m_2}q^{-3m_2}(q^3)^{m_2^2}\sum_{m_1=-\infty}^{\infty}q_1^{m_1}q^{m_1}q^{m_1^2}\\
=&\prod_{m=1}^{\infty}(1-(q^3)^{2m})(1+q_1q_2^2(q^3)^{2m-1})(1+(q_1q_2^2)^{-1}(q^3)^{2m-1})\cdot\\*
&\prod_{m=1}^{\infty}(1-q^{2m})(1+q_1q^{2m-1})(1+q_1^{-1}q^{2m-1})\\*
&+q_2^{-1}q\prod_{m=1}(1-(q^3)^{2m})(1+q_1q_2^2(q^3)^{2m-2})(1+(q_1q_2^2)^{-1}(q^3)^{2m})\cdot \\*
&\prod_{m=1}^{\infty}(1-q^{2m})(1+q_1q^{2m})(1+q_1^{-1}q^{2m-2})\;,
\end{align*}
where at the last equality we have used (\ref{eq:jactr2}) again.

Dividing this with $\prod_{m=1}^{\infty}(1-q^m)^{3}$ and using that
\[ q_1(1+q_0q_2q^{-1})=(1+q_1) \]
gives (\ref{eq:inf3}) after cancellations.

%\begin{multline*}Z_0(\underline{q})=\\ \prod_{m=1}^\infty %\frac{(1-q^{6m})(1+q_{1}^{-1}q_2^2q^{6m-3})(1+q_{1}q_2^{-2}q^{6m-3})(1+q_{1}q^{2m-1})(1+q_{1}^{-1}q^{2m-1})}{(1-q^m)^2(1-q^{2m+1})} \\
%+ q_2 q\prod_{m=1}^\infty %\frac{(1-q^{6m})(1+q_{1}^{-1}q_2^2q^{6m})(1+q_{1}q_2^{-2}q^{6m-6})(1+q_{1}q^{2m})(1+q_{1}^{-1}q^{2m-2})}{(1-q^m)^2(1-q^{2m+1})}
%\;.
%\end{multline*}

\section{Final comments}

The proof of Theorem \ref{thm:coldiaggen} in \cite{fujii2005combinatorial} is based on the decomposition of Young diagrams into cores and quotients as developed in \cite[Section 2.7]{james1981representation}. The \emph{$n$-core} of a diagonally colored Young diagram is the diagonally colored Young diagram obtained by successively removing border strips of length $n$, until this is no longer possible. Here a {\em border strip} is a skew Young diagram which does not contain $2 \times 2$ blocks and 
which contains exactly one $c$-labelled block for all labels $c \in C$. The removal of border strips from a diagonally colored Young diagram can be traced on another combinatorial object, the \emph{abacus}. The abacus for $n$ colors consists of \emph{rulers} corresponding to the residue classes in $C=\SZ/n\SZ$.  The $i$-th ruler consists of integers in the $i$-th residue class modulo $n$ in increasing order from top to bottom. Several \emph{beads} are placed on these rulers, at most one on each integer. %in the limit upwards each position is filled up with beads, and in the limit downwards there are no beads. 
In particular, to a Young diagram corresponding to the partition $\lambda=(\lambda_1,\dots,\lambda_k)$ place a bead in position $\lambda_i-i+1$ for all $i$, interpreting $\lambda_i$ as 0 for $i>k$. 
The removal of a border strip from the Young diagram corresponds to moving a bead up on one of the rulers. It turns out that shifting of beads on different rulers is independent from each other. In this way, the core of a partition corresponds to the bead configuration in which all the beads are shifted up as much as possible. Let us denote by $\mathcal{C}$ the set of $n$-core partitions. %, and
It can be shown that the configuration of the beads on a ruler is described by a partition. The collection of these is called the \emph{$n$-quotient}. Hence, we get a bijection
\begin{equation}\label{eq:corequot} {\mathcal P} \longleftrightarrow {\mathcal C} \times {\mathcal P}^{n}, \end{equation}
 compatible with the diagonal coloring.

Given an $n$-core, one can read the $n$ runners of its abacus representation 
separately. The lowest bead on the $i$-th ruler will have the position $j_i$, which is negative if the shift is 
toward the negative positions (upwards), and positive otherwise. These numbers are the same that appear in our proof in \eqref{eq:z0coeff}, and they have to satisfy 
$\sum_{i=0}^{n-1} j_i=0$.
The set $\{j_0,\dots,j_{n-1}\}$ completely determines the core Young diagram, so we get a bijection
\begin{equation*}\label{typeA-cores} 
{\mathcal C} \longleftrightarrow \left\{\sum_{i=0}^{n-1} j_i=0\right\}\subset\SZ^{n}.\end{equation*}
%We will represent an $(n+1)$-core partition by the corresponding $(n+1)$-tuple
%$\underline{a}=(a_0,\dots,a_{n})$. 

%On the other hand, for an arbitrary partition, on each runner we have a partition up to shift, 
%so we get a bijection
%\[ {\mathcal Z}_\Delta \longleftrightarrow {\mathcal C}_\Delta \times {\mathcal P}^{n+1}. \]
The decomposition \eqref{eq:corequot} reveals the structure of the
formula of Theorem \ref{thm:coldiaggen}; the first term is the generating series of $n$-tuples of 
(uncolored) partitions, whereas the second term is exactly a sum over 
$\underline{j}=(j_0,\dots,j_{n-1})\in {\mathcal C}$, i.e. the multi variable generating series of the $n$-core Young diagrams (see the correspondence \eqref{eq:mjcorr} below). 

All these, and our new proof for Theorem \ref{thm:coldiaggen} also imply that the colored F-partitions introduced in \ref{subsec:colfpart} above as one of our main tools %which are introduced in \ref{subsec:colfpart} above as one of our main tools, %
may have a deeper connection to the core/quotient decomposition \eqref{eq:corequot}. This connection can be the subject of further investigations.
%The multiweight of a core partition corresponding to an 
%element $\underline{j}$ is given by the quadratic expression $Q(\underline{j})$ in the
%exponent of the second term. 
%For more details, see Bijections 1-2 in~\cite[\S2]{garvan1990cranks}. 

\subsection*{Acknowledgement:} The author thanks to Balázs Szendrői, András Némethi and to the anonymous referee for numerous comments on earlier versions of this manuscript. The author was partially supported by the \emph{Lend\"ulet program} of the Hungarian Academy of Sciences and by the ERC Advanced Grant LDTBud (awarded to Andr\'as Stipsicz).

%\bibliographystyle{amsplain}
%\bibliography{frobpart}

\providecommand{\bysame}{\leavevmode\hbox to3em{\hrulefill}\thinspace}
\providecommand{\MR}{\relax\ifhmode\unskip\space\fi MR }
\providecommand{\MRhref}[2]{%
  \href{http://www.ams.org/mathscinet-getitem?mr=#1}{#2}
}

\providecommand{\href}[2]{#2}

\begin{comment}
\providecommand{\bysame}{\leavevmode\hbox to3em{\hrulefill}\thinspace}
\providecommand{\MR}{\relax\ifhmode\unskip\space\fi MR }
\providecommand{\MRhref}[2]{%
  \href{http://www.ams.org/mathscinet-getitem?mr=#1}{#2}
}
\providecommand{\href}[2]{#2}

\end{comment}

\end{document}